\documentclass[12pt]{amsart}%,reqno
\usepackage{amssymb,amsmath,amsthm}
\usepackage{a4}
\usepackage{color}
\usepackage{enumitem}
\usepackage[colorlinks=true,
linkcolor=webgreen,
filecolor=webbrown,
citecolor=webgreen]{hyperref}
\definecolor{webgreen}{rgb}{0,0,1}%{1,0.0,.6}
\definecolor{recrown}{rgb}{1,.2,.6}
\usepackage{fullpage}
\usepackage{orcidlink}
\begin{document}
\newtheorem{theorem}{Theorem}
\newtheorem{corollary}[theorem]{Corollary}
\newtheorem{lemma}[theorem]{Lemma}
\theoremstyle{definition}
\newtheorem{example}{\bf Example}
\theoremstyle{theorem}
\newtheorem{conjecture}[theorem]{Conjecture}
\newtheorem{thmx}{\bf Theorem}
\renewcommand{\thethmx}{\text{\Alph{thmx}}}% "letter-numbered" theorems
\newtheorem{lemmax}{\bf Lemma}
\renewcommand{\thelemmax}{\text{\Alph{lemmax}}}% "
%\leftmargin=.5in
%\rightmargin=0.5in
%\textwidth=6truein
%\textheight=11.6truein
\hoffset=-0cm
%\voffset=+-2cm
\theoremstyle{definition}
\newtheorem*{definition}{Definition}
\theoremstyle{remark}
\newtheorem*{remark}{\bf Remark}
\theoremstyle{remark}
\newtheorem*{remarks}{\bf Remarks}
%\title{\bf On a Lemma of Singh and Kumar}
\parindent=0cm
\title{\bf On irreducible factors of polynomials over integers}
\author{Rishu Garg$^{1}$ {\large \orcidlink{0009-0008-2348-8340}}}
\author{Jitender Singh$^{2,\dagger}$ {\large \orcidlink{0000-0003-3706-8239}}}
\address[1,2]{Department of Mathematics,
Guru Nanak Dev University, Amritsar-143005, India\newline %\linebreak
{\tt jitender.math@gndu.ac.in}, {\tt rishugarg128@gmail.com}}
%\author{I. E. Sarris$^3$ {\large \orcidlink{0000-0002-6542-0490}}}
%\address[3]{University of West Attica; {\tt sarris@uniwa.gr}}
\markright{}
\date{}
\footnotetext[2]{Corresponding author email(s): {\tt jitender.math@gndu.ac.in}

2020MSC: {Primary 12E05; 11C08}\\

}
\maketitle
\newcommand{\K}{\mathbb{K}}
\begin{abstract}
In this paper, we obtain several new factorization results for certain classes of polynomials having integer coefficients. In doing so, we use the information about prime factorization of the value taken up by such polynomials and their higher order formal derivatives at sufficiently large integer arguments. If a lower bound for the minimum possible degree of a factor of such a  polynomial is known a priori, then the integer argument becomes significantly smaller, which makes the underlying factorization result easier to apply. A result on explicit lower degree factor bound for the classes of polynomials considered in this paper is also proved via Newton polygons.
\end{abstract}
\section{Introduction}
For a polynomial $f=a_0+a_1 z+\cdots+a_nz^n\in \mathbb{Z}[z]$ with $a_0a_n\neq 0$, we define the polynomial $s_i\in \mathbb{Z}[z]$ by
\begin{eqnarray*}
s_i(z)=f^{(i)}(z)/i!,~i=0,1,\ldots, n,
\end{eqnarray*}
where $f^{(0)}(z)=f(z)$ and $f^{(i)}(z)$ for $i\geq 1$ denote the formal $i$th order derivative of $f$ with respect to $z$. The height $h_f$ of the polynomial $f$ is defined as the number
\begin{eqnarray*}
h_f:=\max_{0\leq i\leq n-1}\{|a_i|/|a_n|\}.
\end{eqnarray*}
%The height $h_f$ of $f$ is related to location of zeros of $f$ in the complex plane.
In \cite{Mu}, a well appreciated inventive fact that all the zeros of $f$  lie in the open disc $|z|<h_f + 1$ in the complex plane supported with the condition that $|f|$ takes a prime value for an integer argument strictly exceeding $h_f+1$ impelled the irreducibility of $f$ in $\mathbb{Z}[z]$. This result was further generalized by Girstmair in \cite{G} waiving off the primality hypothesis for primitive polynomials $f$ satisfying $f(m)=\pm pd$ for a positive integer $d$ a prime number $p\nmid d$ and $m$ strictly exceeding $h_f + d$.  In \cite{J-S-2}, Girstmair's irreducibility criterion was further generalized and the following result was proved.
\begin{thmx}[\cite{J-S-2}]\label{th:A}
Let $f=a_0+a_1 z+\cdots+a_nz^n\in \Bbb{Z}[z]$ be a primitive polynomial with $a_0a_n\neq 0$. Suppose there exist natural numbers $m$, $d$, $k$, $j\leq n$, and a prime $p\nmid d$ such that $m\geq h_f+1+d$, $f(m)=\pm p^k d$, $\gcd(k,j)=1$, $p^k$ divides $s_i(m)$ for each index $i=1,\ldots,j-1$, and for $k>1$, also $p$ does not divide $s_j(m)$. Then the polynomial $f$ is irreducible in $\Bbb{Z}[z]$.
\end{thmx}
Gisrtmair's criterion corresponds to the case $k=1$ of Theorem \ref{th:A}. In fact Girstmair's criterion has been further generalized recently in \cite{Singh24} and \cite{R-J-S} enabling several classes of irreducible polynomials over integers.

In this paper, we further generalize aforementioned irreducibility results by imposing conditions on one or higher order derivatives of $f$ evaluated at sufficiently large integer argument $m$, where some of the criteria proved in this paper derive motivation from the classical irreducibility criterion of Dumas \cite{D}. To proceed, we recall that for any prime $p$ and nonzero integer $a$, $v_p(a)$ denote the largest nonnegative integer for which $p^{v_p(a)}$ divides $a$, and we define $v_p(0)=\infty$. We also recall that a polynomial $f\in \mathbb{Z}[z]$ is primitive if greatest common divisor of all the coefficients of $f$ is $1$.

First of our main results is the following factorization theorem.
\begin{theorem}\label{th:1}
Let $f=a_0+a_1 z +\cdots+a_nz^n\in \mathbb{Z}[z]$ be a primitive polynomial with $a_0a_n\neq 0$. Suppose there exists a positive integer $m\geq h_f+2$ for which $f(m)=\pm p_1^{k_1}\cdots p_r^{k_r}$,  where $p_1,\ldots, p_r$ are  primes which are all distinct for $r\geq 2$. For each $i=1,\ldots, r$, let $j_i\leq n$ be a positive integer for which the following conditions are satisfied.
\begin{enumerate}
\item[(i)] $v_{p_i}(s_{j_i}(m))=0$, and $\gcd(k_i,j_i)=1$,
\item[(ii)] If $j_i>1$, then $\frac{k_i}{j_i}<\frac{v_{p_i}(s_t(m))}{j_i-t}$ for each $t=1,\ldots,j_i-1$.
\end{enumerate}
Then $f$ is a product of at most $r$ irreducible polynomials in $\mathbb{Z}[z]$. In particular, if $r=1$, then $f$ is irreducible.
\end{theorem}
Theorem \ref{th:1} in particular, for $r=1$, yields a generalization of Theorem \ref{th:A} proved in \cite[Theorem 1]{R-J-S}. Further, the upper bound mentioned in Theorem \ref{th:1} is best possible in the sense that there exist polynomials over integers attaining the said bound. For example, the polynomial
\begin{eqnarray*}
f=64+56z^2+14z^4+z^6
\end{eqnarray*}
satisfies the hypothesis of Theorem \ref{th:1} for $m=117>66=h_f+2$, since here, we have
\begin{eqnarray*}
s_0(117)=13691\times 13693\times 13697,
\end{eqnarray*}
which is a product of three distinct primes $p_1=13691$, $p_2=13693$, and $p_3=13697$ so that $r=3$ and $k_1=k_2=k_3=1$. Further, $s_i(117)$ is coprime to $p_t$ for each $i=1,\ldots,6$ and $t=1,2,3$. Consequently, we have $j_1=j_2=j_3=1$. Thus, the polynomial $f$ is a product of at most $r=3$ irreducible factors in $\mathbb{Z}[z]$, where a direct computation reveals that
\begin{eqnarray*}
64+56z^2+14z^4+z^6=(2+z^2)(4+z^2)(8+z^2),
\end{eqnarray*}
which is a product of exactly three irreducible polynomials in $\mathbb{Z}[z]$. Note that none of Theorem \ref{th:A} and  \cite[Theorem 1]{R-J-S} is applicable for deducing information about factorization of the polynomial $64+56z^2+14z^4+z^6$.
\begin{theorem}\label{th:2}
Let $f=a_0+a_1 z +\cdots+a_nz^n\in \mathbb{Z}[z]$ be a primitive polynomial with $a_0a_n\neq 0$. Suppose there exists a positive integer $m\geq h_f+2$ for which $s_n(m)=\pm p_1^{k_1}\cdots p_r^{k_r}$,  where $p_1,\ldots, p_r$ are  primes which are all distinct for $r\geq 2$.  For each $i=1,\ldots, r$, let $j_i\leq n$ be a positive integer for which the following conditions are satisfied.
\begin{enumerate}
\item[(i)] $v_{p_i}(s_{n-j_i}(m))=0$, and $\gcd(k_i,j_i)=1$,
\item[(ii)] If $j_i>1$, then $\frac{k_i}{j_i}<\frac{v_{p_i}(s_{n-t}(m))}{j_i-t}$ for each $t=1,\ldots,j_i-1$,
\item[(iii)] $|s_0(m)/q|\leq |s_n(m)|$, where $q$ is the smallest prime divisor of $|s_0(m)|$.
\end{enumerate}
Then $f$ is a product of at most $r$ irreducible polynomials in $\mathbb{Z}[z]$. In particular, if $r=1$, then $f$ is irreducible.
\end{theorem}
The upper bound on number of irreducible factors of $f$ as mentioned in Theorem \ref{th:2} is attained by the polynomial
\begin{eqnarray*}
f=81 + 1782 z^2 + 9797 z^4
\end{eqnarray*}
for $m=8>h_f+2$, since here, we have $s_0(8)=6217\times 6473$, where $6217$ and $6473$ are prime numbers so that $q=6217$ and $s_0(8)/q=6473<9797=s_4(8)$ holds. Further, we have $s_4(8)=9797=97\times 101$, where $p_1=97$ and $p_2=101$ are prime numbers so that $r=2$ and $k_1=1=k_2$. Since $s_3(8)=2^5\times 97\times 101$, $s_2(8)=2\times 3\times 5\times 7\times 17923$, where $17923$ is also a prime, it follows that $s_2(8)$ is coprime to $97\times 101$, which shows that $j_1=j_2=2$. We then have
\begin{eqnarray*}
\frac{k_i}{j_i}=\frac{1}{2}<1=\frac{v_{p_i}(s_3(8))}{j_i-1},~i=1,2.
\end{eqnarray*}
We conclude that $f$ is a product of at most two irreducible polynomials in $\mathbb{Z}[z]$, where we note that $f$ factors as
 \begin{eqnarray*}
f(z)=81 + 1782 z^2 + 9797 z^4=(9 + 97 z^2) (9 + 101 z^2),
\end{eqnarray*}
which is a product of exactly two irreducible polynomials in $\mathbb{Z}[z]$.

For a nonconstant polynomial $f$ having integer coefficients, let $\Delta_f\leq \deg f$ be a natural number such that $f$ has no factor of degree less than $\Delta_f$ in $\mathbb{Z}[z]$. It follows that if
\begin{eqnarray*}
\deg f\leq 2\Delta_f,
\end{eqnarray*}
then the polynomial $f$ is irreducible in $\mathbb{Z}[z]$. The following result will be useful in computing $\Delta_f$ for a large class of polynomials having integer coefficients, which also generalizes a result of the paper \cite[Corollary~ 3.8]{Kh23}.
	\begin{theorem}\label{th:5}
		Let $f=a_0+a_1z+\cdots +a_nz^n \in \mathbb{Z}[z]$ with $a_0a_n\neq 0$. Suppose there exists a prime number  $p$ and a positive integer $j\leq n$ for which the following hold.
		\begin{enumerate}[label=$(\roman*)$]
			\item[$(i)$] $v_p(a_j)=0$,
			\item[$(ii)$] $\frac{v_p(a_0)}{j}\leq \frac{v_p(a_i)}{j-i}$ for each $i=0,1,\ldots,j-1$,
			\item[(iii)] If $j<n$, then $\frac{v_p(a_n)}{n-j} \leq \frac{v_p(a_i)}{i-j}$ for each $i=j+1,j+2,\ldots,n-1$.
		\end{enumerate}
Let $d_1=\gcd(v_p(a_0),j)$, and if $j<n$, then we let $d_2=\gcd(v_p(a_n),n-j)$. Define
\begin{eqnarray*}
k_f=\begin{cases}\min\{j/d_1, (n-j)/d_2\},~&~\text{if}~j<n;\\
n/d_1,~&~\text{if}~j=n.
\end{cases}
\end{eqnarray*}
		Then any irreducible factor of $f$ in $\mathbb{Z}[z]$ has degree at least $k_f$.
	\end{theorem}
Note that a polynomial $f$ over integers which satisfies the hypothesis of Theorem \ref{th:5} also satisfies the following:
\begin{eqnarray*}
\Delta_f=k_f.
\end{eqnarray*}
Motivated by a result proved in the paper \cite[Proposition 1]{JA}, we may have the following more general result.
\begin{theorem}\label{th:3}
Let $f=a_0+a_1z+ \cdots +a_nz^n \in \mathbb{Z}[z]$ be a primitive polynomial with $a_0a_n\neq 0$.  Suppose there exist natural numbers $m$, $d$, $k$, and a prime $p\nmid d$  with
\begin{eqnarray*}
(m-1-h_f)^{\Delta_f}\geq d,
\end{eqnarray*}
such that $s_0(m)=\pm p^kd$. Then  the polynomial $f$ is a product of at most
\begin{eqnarray*}
\min_{0\leq i\leq n}\{i+v_p(s_i(m))\}
\end{eqnarray*}
irreducible factors in $\mathbb{Z}[z]$. In particular, if $k=1$, or $p\nmid s_1(m)$, then $f$ is irreducible.
\end{theorem}
The case $k=1=d$ and $j=n$ of Theorem \ref{th:3} is precisely the irreducibility result proved in the paper \cite[Proposition 1]{JA}.

As before, the upper bound mentioned in Theorem \ref{th:3} is best possible in the sense that explicit examples of polynomials attaining this bound exist. In view of this, we consider the polynomial
 \begin{eqnarray*}
 f=-2-4z+3z^2-2z^3+2z^4 \in \mathbb{Z}[z],
 \end{eqnarray*}
for which we have $h_{f}=2$. First observe that $f$ satisfies the hypothesis of Theorem \ref{th:5} for the prime number $2$, $j=2$, $d_1=d_2=1$.  Consequently, each irreducible factor of $f$ in $\mathbb{Z}[z]$ has degree at least $\min\{2/1,2/1\}=2$, that is, $\Delta_{f}=2$. Now applying Theorem \ref{th:3} to $f$ by taking $m=14$, we find that
 $s_0(m)=11^3\cdot 54$, which shows that $p=11$, $k=3$, and $d=54$, where
\begin{eqnarray*}
(m-1-h_{f})^{\Delta_{f}}=(14-1-2)^2=121>54=d.
\end{eqnarray*}
Thus, the polynomial $f$ is a product of at most $\min\{0+3,1+1,2+0,3+1,4+0\}=2$ irreducible factors in $\mathbb{Z}[z]$, where we find that $f$ factors as
		\begin{eqnarray*}
			f(z)=-2-4z+3z^2-2z^3+2z^4 =(2+z^2)(-1-2z+2z^2),
		\end{eqnarray*}
which is a product of exactly two irreducible polynomials over integers.

If instead the information about factorization of the leading coefficient of the underlying polynomial is used, then we may arrive at the following factorization result.
\begin{theorem}\label{th:4}
Let $f = a_0 + a_1z +\cdots+a_n z^n\in \Bbb{Z}[z]$ be a primitive polynomial with $a_0a_n\neq 0$.   Suppose there exist natural numbers $m$, $d$, $k$, and a prime $p\nmid d$ with
\begin{eqnarray*}
(m-1-h_f)^{\Delta_f}\geq d,
\end{eqnarray*}
such that $s_n(m)=\pm p^kd,~|s_0(m)/q|\leq |s_n(m)|$, where $q$ is the smallest prime divisor of $s_0(m)$. Then the polynomial $f$ is a product of at most
\begin{eqnarray*}
\min_{1\leq i\leq n}\{k,i+v_p(s_{n-i}(m))\}
\end{eqnarray*}
irreducible factors in $\mathbb{Z}[z]$. In particular, if $k=1$, or $p\nmid s_{n-1}(m)$, then $f$ is irreducible.
\end{theorem}
The upper bound as mentioned in Theorem \ref{th:4} is taken up by the polynomial
\begin{eqnarray*}
f=9-36z+54z^2-2094z^3+4125z^4-2058z^5+117649z^6,
\end{eqnarray*}
where  $h_{f}<1$. To see this, we first observe that $f$ satisfies the hypotheses of Theorem \ref{th:5} for the prime number $3$, $j=n=6$, and $d_1=2$, which yields $\Delta_{f}=3$. Now  if we take $m=7$, then we find that $s_6(7)=7^6$, and so, we have $p=7$, $k_1=6$, and $d=1$. Further, $s_0(7)=117541^2$, and so, $q=117541$, and we have $s_0(7)/q =117541<117649=s_6(7)$. Since  $(m-1-h_{f})^{\Delta_{f}}>(7-1-1)^3=125>1$, by Theorem \ref{th:4}, we deduce that the polynomial $f$ is a product of at most $\min\{6,1+3,2+0,3+0,4+0,5+0,6+0\}=2$ irreducible factors in $\mathbb{Z}[z]$, where a direct computation confirms that
		\begin{eqnarray*}
			f(z)=9-36z+54z^2-2094z^3+4125z^4-2058z^5+117649z^6 = (-3+6z-3z^2+343z^3)^2,
		\end{eqnarray*}
which is a product of exactly two irreducible polynomials in $\mathbb{Z}[z]$.
\section{Examples}
Below we exhibit some explicit examples of polynomials over integers whose factorization properties are deducible from Theorems \ref{th:1}-\ref{th:4}.
\begin{example} The polynomial
\begin{eqnarray*}
F_1=1287 + 3168 z^2 - 3528 z^3 + 1936 z^4 - 4312 z^5 + 2401 z^6\in \mathbb{Z}[z]
\end{eqnarray*}
satisfies the hypothesis of Theorem \ref{th:1} for $m=4>h_{F_1}+2$, since here $s_0(4)=2393\times 2399$ is a product of the two primes $p=2393$ and $p_2=2399$, and so, $r=2$, $k_1=1=k_2$. Further, we have $s_1(4)=2^7\times 5^3\times 599$, which is coprime to $s_0(4)$,and so,  $j_1=j_2=1$. So, the polynomial $F_1$ is a product of at most two irreducible factors in $\mathbb{Z}[z]$, where we find that the polynomial $F_1$ is a product of the two irreducible polynomials $39+44z^2-49z^3$ and $33+44z^2-49z^3$. The same conclusion for the polynomial $F_1$ was obtained in the paper \cite[Theorem 3]{R-J-S}.
\end{example}
\begin{example}
Consider the polynomial
\begin{eqnarray*}
F_2=4 + 120 z^2 + 899 z^4,
\end{eqnarray*}
which satisfies the hypothesis of Theorem \ref{th:2} for $m=3>h_{F_2}+2$ so that $s_0(3)=263\times 281$, where each of $263$ and $281$ is a prime so that $q=263$, and $s_0(3)/q=281<899=s_4(3)$. Since $s_4(3)=899=29\times 31$, we have $p=29$, $p_2=31$, $k_1=k_2=1$, and since $v_{p_i}(s_3(3))=1$ for each $i=1,2$ and $s_2(3)=2\times 3\times 8111$ which is coprime to $s_4(3)$, it follows that $j_1=j_2=2$. Finally, we have $k_i/j_i=1/2<1=v_{p_i}(s_{4-i}(3))/(2-1)$ for each $i=1,2$. Thus, the polynomial $F_2$ is a product of at most $2$ irreducible factors in $\mathbb{Z}[z]$, where we find that
 \begin{eqnarray*}
F_2(z)=(2 + 29 z^2) (2 + 31 z^2),
\end{eqnarray*}
which is a product of exactly two irreducible polynomials in $\mathbb{Z}[z]$.
\end{example}
\begin{example}
Now consider the polynomial
\begin{eqnarray*}
F_3=4-16z+32z^2+4z^3-56z^4+72z^5+81z^6,
\end{eqnarray*}
which satisfies the hypotheses of Theorem \ref{th:5} for the prime number $2$, $j=n=6$, $d_1=2$, where we note that  $h_{F_3}=72/81<1$. So, we have $\Delta_{F_3}=3$. Observe that the polynomial $F_3$ satisfies the hypothesis of Theorem \ref{th:3} for  		
$m=4$ so that $s_0(4)=313^2\cdot4$, where $p=313$, $k=2$, and $d=4$, where
\begin{eqnarray*}
(m-1-h_{F_3})^{\Delta_{F_3}}>(4-1-1)^3=8>4=d.
\end{eqnarray*}
We deduce that the polynomial $F_3$ is a product of at most $\min\{0+2,1+1,2+0,3+0,4+0\}=2$ irreducible factors in $\mathbb{Z}[z]$, where we find that
		\begin{eqnarray*}
			F_3(z)=(2-4z+4z^2+9z^3)^2,
		\end{eqnarray*}
which is a product of exactly two irreducible factors in $\mathbb{Z}[z]$.
	\end{example}
\begin{example}The polynomial
\begin{eqnarray*}
F_4=2-2z+2z^2-375z^3+100z^4-100z^5+100z^6-18750z^7
\end{eqnarray*}
satisfies the hypotheses of Theorem \ref{th:5} for the prime number $2$, $j=3$, $d_1=1$, and $d_2=1$, and so, we have $\Delta_{F_4}=3$. Now observe that the polynomial $F_4$ satisfies the hypothesis of Theorem \ref{th:4} for $m=3>h_{F_4}+1$, since $s_7(3)=5^5\cdot 6$ so that $p=5$, $k_1=5$, and $d=6$. Further, $s_0(3)=-4051\cdot 10111$, where $4051$ and $10111$ are primes so that $q=4051$ and we have  $|s_0(3)/q| =10111<18750=|s_7(3)|$. Since $(m-1-h_{F_4})^{\Delta_{F_4}}=(3-1-0.02)^3\approx 7.76>6$, by Theorem \ref{th:4}, we deduce that the polynomial $F_4$ is a product of at most $\min\{5,1+2,2+2,3+2,4+2,5+0,6+1,7+0\}=3$ irreducible factors in $\mathbb{Z}[z]$, where we find that $F_4$ factors as
	\begin{eqnarray*}
		F_4(z)=(2-2z+2z^2-375z^3)(1+50z^4),
	\end{eqnarray*}
which is product of exactly two irreducible factors.
	\end{example}
\section{Proofs}
The following result of the paper \cite[Lemma 9]{Singh25} is indispensable for the proof of Theorem \ref{th:1}.
\begin{lemmax}\label{L:1}
Let $g=s_0 + s_1z + \cdots + s_nz^n \in \mathbb{Z}[z]$ be a polynomial of degree $n$ such that there exists an index  $j$ with $1\leq j\leq n$ for which the following hold.
\begin{enumerate}[label=$(\roman*)$]
\item $v_p(s_j)=0$,
\item If $j>1$, then $\frac{v_p(s_0)}{j}<\frac{v_p(s_i)}{j-i}$ for each $i=1,\ldots, j-1$,
\item $\gcd(v_p(s_0),~j) = 1$.
\end{enumerate}
Then any factorization $g(z)=g_1(z)g_2(z)$ of the polynomial $g$ in $\mathbb{Z}[z]$ satisfies $v_p(g_1(0))=0$, or $v_p(g_2(0))=0$.
\end{lemmax}
The following lemma will also we utilized in the sequel for proving our results.
\begin{lemma}\label{L:2}
Let $f=a_0+a_1z+\cdots+a_nz^n\in \mathbb{Z}[z]$ be a primitive polynomial with $a_0a_n\neq 0$. Let $m$ and $d$ be positive integers such that $m\geq h_f+1+d$. If $g(z)=f(m+z)$, then $g\in \mathbb{Z}[z]$ is such that each zero $\theta$ of $g$ in $\mathbb{C}$ satisfies $|\theta|>d$.
\end{lemma}
\begin{proof}[\bf Proof of Lemma \ref{L:2}]
We have $f(m+\theta)=g(\theta)=0$, which shows that $m+\theta$ is a zero of the polynomial $f$. Since all zeros of $f$ lie within the open disc $|z|<h_f+1$ centered at the origin in the complex plane, it follows that
\begin{eqnarray*}
h_f+1>|m+\theta|\geq m-|\theta|\geq h_f+1+d-|\theta|,
\end{eqnarray*}
which shows that $|\theta|>d$, as desired.
\end{proof}
We will now prove Theorem \ref{th:1} as follows.
\begin{proof}[\bf Proof of Theorem \ref{th:1}] Define $g(z)=f(m+z)$ so that $g=s_0(m)+s_1(m)z+\cdots+s_n(m)z^n\in \mathbb{Z}[z]$. Note that $\gcd(s_0(m),\ldots, s_n(m))=1$, since $v_{p_i}(s_{j_i}(m))=0$ for all $i=1,\ldots, r$, and as a consequence of this, the polynomial $g$ is primitive. We further note that the number of irreducible factors of $g$ and $f$ in $\mathbb{Z}[z]$ is the same. In view of this, it will be sufficient to establish that the number of irreducible factors of $g$ over integers is at most $r$. We proceed as follows.  Let $g(z)=g_1(z)\cdots g_N(z)$ be a factorization of the polynomial $g$ into a product of $N$ irreducible factors $g_1,\ldots,g_N$ in $\mathbb{Z}[z]$. We suppose on the contrary that $N>r$. In view of the fact that
\begin{eqnarray*}
p_1^{k_1}\cdots p_r^{k_r}=|f(m)|=|g(0)|=|g_1(0)|\cdots |g_N(0)|,
\end{eqnarray*}
and Lemma \ref{L:2} for $d=1$, we find that  $|g_j(0)|>1$ for each $j=1,\ldots, N$. This on applying the pigeon-hole principle tells us that there exists an index $i\in \{1,\ldots,r\}$ and two distinct indices $\ell_1$ and $\ell_2$ in the set $\{1,\ldots, N\}$ for which $p_i$ divides $g_{\ell_t}(0)=s_{\ell_t}(m)$ for each $t=1,2$. We then have $v_{p_i}(g_{\ell_t}(0))>0$ for each $t=1,2$. Now if we take a partition of the set $\{1,\ldots, N\}$ into two disjoint sets $P_1$ and $P_2$ such that $\ell_t\in P_t$  for $t=1,2$, then we may have the following factorization of $g$.
 \begin{eqnarray*}
g(z)=\Bigl(\prod_{a\in P_1}g_a(z)\Bigr)\Bigl(\prod_{b\in P_2}g_b(z)\Bigr),
  \end{eqnarray*}
  where $G_t(z)=\prod_{a\in P_t}g_a(z)$ is a polynomial  in $\mathbb{Z}[z]$ for each $t=1,2$, such that $g(z)=G_1(z)G_2(z)$ with $v_{p_i}(G_t(0))\geq v_{p_i}(g_{\ell_t}(0))>0$ for each $t=1,2$. This contradicts Lemma \ref{L:1}.
\end{proof}
To prove Theorem \ref{th:2}, we first prove the following factorization result, which may be of independent interest as well.
\begin{lemma}\label{L:4}
Let $g=s_0+s_1 z +\cdots+s_nz^n\in \mathbb{Z}[z]$ be a primitive polynomial such that $s_0s_n\neq 0$ and all zeros of $g$ lie in the region $|z|>1$ in the complex plane. Suppose that  $s_n=\pm p_1^{k_1}\cdots p_r^{k_r}$ for primes $p_1,\ldots, p_r$  which are all distinct for $r\geq 2$ and $k_1,\ldots,k_r$ are all positive integers. For each $i=1,\ldots, r$, let $j_i\leq n$ be a positive integer for which the following conditions are satisfied.
\begin{enumerate}
\item[(i)] $v_{p_i}(s_{n-j_i})=0$, and $\gcd(k_i,j_i)=1$,
\item[(ii)] If $j_i>1$, then $\frac{k_i}{j_i}<\frac{v_{p_i}(s_{n-t})}{j_i-t}$ for each $t=1,\ldots,j_i-1$,
\item[(iii)] $|s_0/q|\leq |s_n|$, where $q$ is the smallest prime divisor of $|s_0|$.
\end{enumerate}
Then $g$ is a product of at most $r$ irreducible polynomials in $\mathbb{Z}[z]$. In particular, if $r=1$, then $g$ is irreducible.
\end{lemma}
The polynomial
\begin{eqnarray*}
g=128 + 120 z^2 - 113 z^4 - 105 z^6
\end{eqnarray*}
satisfies the hypothesis of Lemma \ref{L:4}, since here, the leading term of $g$ is $s_6=105=3\times 5\times 7$, which
 is a product of $r=3$ primes so that $k_1=k_2=k_3=1$, and for $q=2$, we have $s_0/q=128/2=64<105=s_6$. Since $\gcd(s_4,105)=1$, we must have $j_1=j_2=j_3=2$ and $k_i/j_i=\frac{1}{2}< \infty=\frac{v_{p_i}(s_{5})}{2-1}$ for each $i=1,2,3$. Thus, the polynomial $g$ is a
product of at most $3$ irreducible factors in $\mathbb{Z}[z]$, where we find that
\begin{eqnarray*}
128 + 120 z^2 - 113 z^4 - 105 z^6=(1 + z^2) (8 + 7 z^2) (16 - 15 z^2),
\end{eqnarray*}
which is a product of exactly 3 irreducible polynomials in $\mathbb{Z}[z]$.
\begin{proof}[\bf Proof of Lemma \ref{L:4}]
Let $g(z)=g_1(z)\cdots g_N(z)$ be a product of $N$ irreducible factors $g_1$, $\ldots$, $g_N$ in $\mathbb{Z}[z]$. Since each zero of $g$ lies inside the region $|z|>1$, it follows that $|s_0|>|s_n|$ so that $q$ exists,  and that $|g_i(0)|>1$ for each $i=1,\ldots,N$. Let $\alpha_i$ be the leading coefficient of $g_i$ for each $i=1,\ldots, N$, so that we have $\alpha_1\cdots  \alpha_N=s_n=p_1^{k_1}\cdots p_r^{k_r}$. Using the hypothesis, we have
\begin{eqnarray*}
|s_0/q| \leq |s_n|&<&|s_n|\Bigl(\frac{|g_1(0)|}{|\alpha_1|}\cdots \frac{|g_N(0)|}{|\alpha_N|}\Bigr)\times \frac{|\alpha_i|}{|g_i(0)|}= |s_0|\frac{|\alpha_i|}{|g_i(0)|},
\end{eqnarray*}
which shows that $|g_i(0)/q|<|\alpha_i|$ for each $i\in\{1,\ldots,N\}$. Since  $|g_i(0)/q|\geq 1$, for each index $i$, we must have $|\alpha_i|>1$ for each $i$. This in view of the fact that $\pm p_1^{k_1}\cdots p_r^{k_r}=s_n=\alpha_1\cdots \alpha_N$ shows that for each $i=1,\ldots,N$, there exists a prime $p_{t_i}$ dividing $\alpha_i$ for some $t_i\in \{1,\ldots,r\}$.
Now suppose on the contrary that $N>r$. Then there exists an index $i\in \{1,\ldots,r\}$ and two distinct indices $\ell_1$ and $\ell_2$ in the set $\{1,\ldots, N\}$ for which $p_i$ divides $\alpha_{\ell_t}$ for each $t=1,2$. Now if we take a partition of the set $\{1,\ldots, N\}$ into two disjoint sets $P_1$ and $P_2$ such that $\ell_t\in P_t$  for $t=1,2$, and if we let   $G(z)=z^{n}g(1/z)$, we find that $G\in \mathbb{Z}[z]$ and that $G$ satisfies the hypothesis of Lemma \ref{L:1} with $G(0)=\alpha_1\cdots \alpha_N$, and
 \begin{eqnarray*}
G(z)=\Bigl(\prod_{a\in P_1}z^{\deg g_a}g_a(1/z)\Bigr)\Bigl(\prod_{b\in P_2}z^{\deg g_b}g_b(1/z)\Bigr),
  \end{eqnarray*}
  where $G_t(z)=\prod_{a\in P_t}z^{\deg g_a}g_a(1/z)$ is a polynomial  in $\mathbb{Z}[z]$ for each $t=1,2$, such that $G(z)=G_1(z)G_2(z)$ with
  \begin{eqnarray*}
v_{p_i}(G_t(0))\geq v_{p_i}(\alpha_{\ell_t})>0,
  \end{eqnarray*}
for each $t=1,2$. This contradicts  Lemma \ref{L:1}.
\end{proof}
Using Lemma \ref{L:4}, we now prove Theorem \ref{th:2} as follows.
\begin{proof}[\bf Proof of Theorem \ref{th:2}] Define $g(z)=f(m+z)$, so that $g=s_0(m)+s_1(m)z+\cdots +s_n(m)z^n$. Using Lemma \ref{L:2} for $d=1$, we find that all zeros of $g$ lie in the region $|z|>1$ in the complex plane. Thus, coefficients of $g$ satisfy the hypothesis of Lemma \ref{L:4}. Consequently, $g$ and hence $f$ is a product of at most $r$ irreducible polynomials in $\mathbb{Z}[z]$.
\end{proof}
To prove Theorem \ref{th:5}, we will make use of the following fundamental result resting on Dumas Theorem \cite{D} on Newton polygons.
\begin{thmx}[Dumas \cite{D}]\label{th:B}
Let $f=a_0+a_1 z+\cdots +a_n z^n\in \mathbb{Z}[z]$ be such that $a_0a_n\neq 0$. Suppose there exist nonconstant polynomials $f_1,\ldots, f_r\in \mathbb{Z}[z]$ such that $f(z)=f_1(z)\cdots f_r(z)$. Let $p$ be a prime number. Let $\{A_i(x_i,y_i) ~|~ 0\leq i\leq k, x_0=0<x_1<\cdots <x_k=\deg f\}$ be the list of $k+1$ lattice points on Newton polygon $NP(f)$ with respect to the prime $p$. For each $s\in \{1,\ldots, k\}$, let $b_s=x_s-x_{s-1}$ and $\{\epsilon_{ts} ~|~ 1\leq t\leq r, 1\leq s\leq k\}=\{0,1\}$. Then for each $s\in \{1,\ldots, k\}$, there exists unique $i\in \{1,\ldots,r\}$ for which $\epsilon_{is}=1$ and that for each $i=1,\ldots, r$, we have
\begin{eqnarray*}
\deg f_i =\epsilon_{i1}b_1+\epsilon_{i2}b_2+\cdots+\epsilon_{ik}b_k.
\end{eqnarray*}
\end{thmx}
The following standard result will also be used in the sequel.
\begin{lemma}\label{L:7}
The number of lattice points lying on the line segment $AB$ joining the points having cartesian coordinates $A(a_1,b_1)$ and $B(a_2,b_2)$ in $\mathbb{R}^2$ including the endpoints  is equal to
\begin{eqnarray*}
1+\gcd(|a_1-a_2|,|b_1-b_2|).
\end{eqnarray*}
\end{lemma}
\begin{proof}[\bf Proof of Theorem \ref{th:5}]
Let $f(z)=f_1(z)f_2(z)\cdots f_N(z)$ be a factorization of $f$ into $N$ irreducible polynomials $f_1,\ldots, f_N$
in $\mathbb{Z}[z]$. We may assume without loss of generality that $f_t$ is nonconstant for each $t=1,\ldots, N$.

First we assume that $j<n$. By the hypothesis, Newton polygon $NP(f)$ of $f$ with respect to $p$ has an edge joining the endpoints $A(0,v_p(a_0))$ and $B(j,0)$ and another edge joining the endpoints $B(j,0)$ and $C(n,v_p(a_n))$  as shown in Fig.~\ref{f1}. By Lemma \ref{L:7}, the total number of lattice points on the edges $AB$  and $BC$ including the endpoints are
\begin{eqnarray*}
1+\gcd(v_p(a_0),j)=1+d_1~\text{and}~1+\gcd(v_p(a_n),n-j)=1+d_2,
\end{eqnarray*}
respectively (see Fig.~\ref{f1}). Now letting
\begin{eqnarray*}
E=\{0,1,2,\ldots, d_1\},~F=\{d_1+1,\ldots, d_1+d_2\},
\end{eqnarray*}
and  denoting by $\chi_E$ and $\chi_F$, the characteristic functions of $E$ and $F$, respectively, we find that the list of all $d_1+d_2+1$ lattice points on $NP(f)$ with respect to $p$ is given by $(x_\ell,y_\ell)$, $\ell=0,1,\ldots, d_1+d_2$, where for each such index $\ell$, we have
\begin{eqnarray*}
x_\ell &=&\frac{j}{d_1}\ell\chi_{E}(\ell)+\Bigl(j+\frac{\ell-d_1}{d_2}(n-j)\Bigr)\chi_{F}(\ell),\\
y_\ell &=&\frac{v_p(a_0)}{d_1}(d_1-\ell)\chi_{E}(\ell)+\frac{v_p(a_n)}{d_2} (\ell-d_1)\chi_{F}(\ell).
\end{eqnarray*}
\begin{figure}[h!!!]
			\centering
			\includegraphics[width=\textwidth]{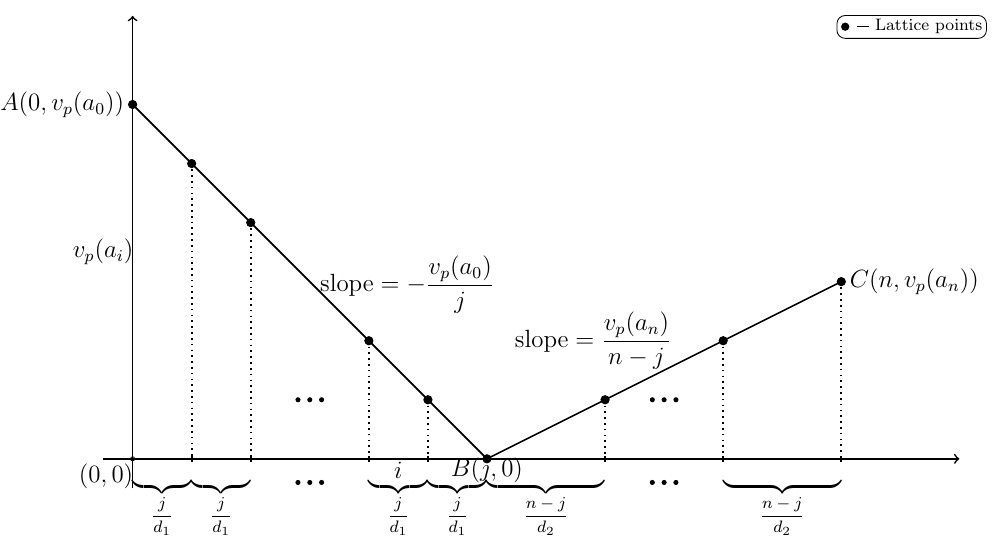}
			\caption{Newton polygon $NP(f)$ of $f$ for the case when $j<n$.}\label{f1}
		\end{figure}
By Theorem \ref{th:B} we find that for each $\ell=0,1,\ldots, k=d_1+d_2$, there exist unique index $i_\ell\in \{1,\ldots,N\}$ such that $\epsilon_{i_\ell \ell}=1$ (so that $\sum_{t=1}^{d_1+d_2}\epsilon_{it}\geq 1$ for each such index $i$), and that for each $i=1,\ldots, N$, we have
\begin{eqnarray*}
\deg f_i&=&\sum_{t=1}^{d_1+d_2}\epsilon_{it} (x_{t}-x_{t-1})\\
&=&\frac{j}{d_1}\sum_{t=1}^{d_1}\epsilon_{it}t+\sum_{t=d_1+1}^{d_1+d_2}\epsilon_{it}\Bigl(j+\frac{t-d_1}{d_2}(n-j)\Bigr)\\
&-&\frac{j}{d_1}\sum_{t=2}^{d_1+1}\epsilon_{it}(t-1)-\sum_{t=d_1+2}^{d_1+d_2}\epsilon_{it}\Bigl(j+\frac{t-d_1-1}{d_2}(n-j)\Bigr)\\
&=&\frac{j}{d_1}\sum_{t=1}^{d_1}\epsilon_{it}+\frac{n-j}{d_2}\sum_{t=d_1+1}^{d_1+d_2}\epsilon_{it}
\geq\min\left\{\frac{j}{d_1},~\frac{n-j}{d_2}\right\}\sum_{t=1}^{d_1+d_2}\epsilon_{it},
\end{eqnarray*}
which in view of the fact that $\sum_{t=1}^{d_1+d_2}\epsilon_{it}\geq 1$ shows that
\begin{eqnarray*}
\deg f_i\geq \min\left\{\frac{j}{d_1},~\frac{n-j}{d_2}\right\},~\text{for each}~i=1,\ldots,N.
\end{eqnarray*}
Now assume that $j=n$. In this case, $NP(f)$ consists only of the line segment $AB$ joining the points $A(0,v_p(a_0))$ and $B(n,0)$ as shown in Fig.~\ref{f2} so that the list of all lattice points on $NP(f)$ consists of the $ d_1+1$ points $(x_\ell,y_\ell)$, $0\leq \ell\leq d_1$, where for each such index $\ell$, we have
\begin{eqnarray*}
x_\ell &=&\frac{n}{d_1}\ell,~y_\ell =\frac{v_p(a_0)}{d_1}(d_1-\ell),
\end{eqnarray*}
\begin{figure}[h!!!]
			\centering
			\includegraphics[width=0.8\textwidth]{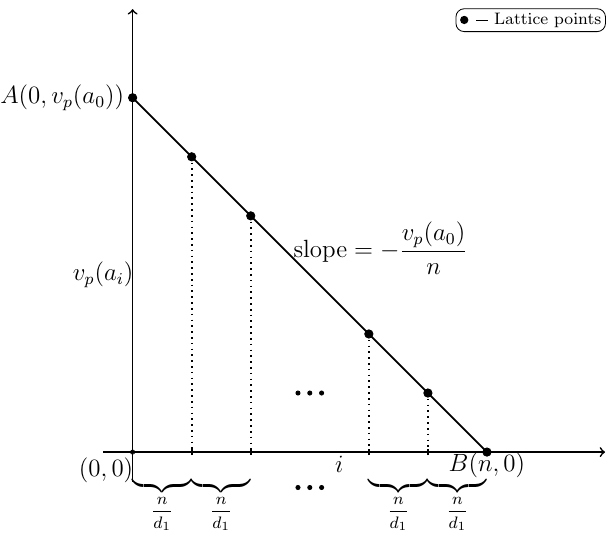}
			\caption{Newton polygon $NP(f)$ of $f$ for the case when $j=n$.}\label{f2}
		\end{figure}
Again by Theorem \ref{th:B} for each $\ell=0,1,\ldots, k=d_1$, there exist unique index $i_\ell\in \{1,\ldots,N\}$ such that $\epsilon_{i_\ell \ell}=1$ (so that $\sum_{t=1}^{d_1}\epsilon_{it}\geq 1$ for each such index $i$), and that for each $i=1,\ldots, N$, we have
\begin{eqnarray*}
\deg f_i&=&\sum_{t=1}^{d_1}\epsilon_{it} (x_{t}-x_{t-1})=\frac{n}{d_1}\sum_{t=1}^{d_1}\epsilon_{it}\geq \frac{n}{d_1},
\end{eqnarray*}
which proves the theorem for the case when $j=n$.
	\end{proof}
The following result will be used in proving Theorem \ref{th:3}.
\begin{lemma}\label{L:3}
Let $g=s_0+s_1z+ \cdots +s_nz^n \in \mathbb{Z}[z]$ be a primitive polynomial with $s_0s_n\neq 0$ and $s_0=\pm p^kd$ for some positive integers $k$ and $d$ and a prime $p\nmid d$ such that all zeros of $g$ lie in the region $|z| > \sqrt[\Delta_f]{d}$ in the complex plane. Then the polynomial $g$ is a product of at most
\begin{eqnarray*}
\min_{0\leq i\leq n}\{i+v_p(s_i)\}
\end{eqnarray*}
irreducible polynomials in $\mathbb{Z}[z]$. In particular, if $k=1$, or  $p\nmid s_1$, then $f$ is irreducible.
\end{lemma}
Consider the polynomial
\begin{eqnarray*}
f=-3+3z+343z^2-126 z^4+126 z^5+14406z^6,
\end{eqnarray*}
which satisfies the hypothesis of Theorem \ref{th:5} for the prime number $3$, $j=2$, $d_1=d_2=1$ so that here, $\Delta_f=\min\{2,4\}=2$.

Now if we consider the polynomial $g$ such that $g(z)=z^6f(1/z)$, then for $g$, we have $s_0=14406=7^4\times 6$ so that $p=7$, $k=4$, and $d=6$. We then have $\sqrt[\Delta_f]{d}=\sqrt{6}$, and we find that
\begin{eqnarray*}
126\sqrt{6}+126 \sqrt{6}^2+343\sqrt{6}^4+3\sqrt{6}^5+3\sqrt{6}^6\approx 14325.2 < 14406.
\end{eqnarray*}
Consequently, each zero of $g$ lies inside the open region $|z|>\sqrt{6}=\sqrt[\Delta_f]{d}$ in the complex plane.
By Lemma \ref{L:3}, the polynomial $g$ (hence $f$) is a product of at most
\begin{eqnarray*}
\min\{0+4,1+1,2+1,3+\infty,4+3,5+0,6+0\}=2
\end{eqnarray*}
irreducible factors in $\mathbb{Z}[z]$, where we note that $g$ factors as
\begin{eqnarray*}
g(z)=14406+126 z-126z^2+343z^4+3z^5-3z^6=(343+3 z-3z^2)(42+z^4),
\end{eqnarray*}
which is a product of exactly two irreducible polynomials in $\mathbb{Z}[z]$.
\begin{proof}[\bf Proof of Lemma \ref{L:3}]
Suppose that  $g(z)=g_1(z)\cdots g_N(z)$ be a product of $N$ irreducible factors $g_1,\hdots ,g_N$ in $\mathbb{Z}[z]$. Then $\Delta_f\leq \deg g_i$ for each $i=1,\ldots, N$. By the hypothesis, we have  $p^kd = |g(0)|= |g_1(0)| \cdots |g_N(0)|$, which shows that $|g_i(0)|\geq 1$ for each $i = 1, \hdots,N$. If $\alpha_i \neq 0$ is the leading coefficient of $g_i$ for each $i=1,\ldots, N$, then we may write $g_i= \alpha_i\prod_{\theta}(z-\theta)$, where the product is over all the zeros of $g_i$. Combining the aforementioned observations together with the hypothesis, we reach at the following conclusion:
\begin{eqnarray*}
|g_i(0)|=|\alpha_i|\prod_{\theta}|\theta| > |\alpha_i|\Bigl(\sqrt[\Delta_f]{d}\Bigr)^{\deg g_i} \geq |\alpha_i|\Bigl(\sqrt[\Delta_f]{d}\Bigr)^{\Delta_f}=|\alpha_i|d\geq d,
\end{eqnarray*} which shows that $|g_i(0)|>d$, and so, $p$ divides $|g_i(0)|$ for every $i = 1, \hdots,N$. This proves that $N \leq k$.

Now let $j\in \{1,\ldots, n\}$. Assume that $v_p(s_j)=J\geq 0$. We may assume without loss of generality that $k-j>J$. If possible, let $N-j>J$. Expressing $g_i$ as $g_i= \sum_{t=0}^{\deg g_i}s_{it}z^t \in \mathbb{Z}[z]$,  we find that
\begin{eqnarray*}
s_j = \sum_{i_1+i_2+\cdots +i_N=j} s_{1i_1}s_{2i_2}\cdots s_{Ni_N},
\end{eqnarray*}
where the indices under the summation sign belong to the set $\{0,1,\ldots, j\}$.  Since $N-j>J \geq 0$, for  $i_1, \ldots, i_N$ in the set $\{0,1,\ldots, j\}$ with $i_1 +\cdots +i_N = j$, one finds that at least $N-j$ of the indices $i_1, \ldots, i_N$ are all equal to zero. Since $g_i(0)=s_{i0}$ and $p$ divides $|g_i(0)|$ for each $i$, it follows that $p^{N-j}$ divides the expression $s_{1i_1}s_{2i_2}\cdots s_{Ni_N}$ for all choices of the indices $i_1,\ldots, i_N$ which satisfy $i_1+\cdots+i_N=j$. Consequently, $p^{N-j}$ divides the sum $\sum_{i_1+i_2+\cdots +i_N=j}s_{1i_1}\cdots s_{Ni_N}=s_j$, which shows that $J=v_p(s_j)\geq N-j>J$, a contradiction.
\end{proof}
\begin{proof}[\bf Proof of Theorem \ref{th:3}] We proceed as in the proof of Theorem \ref{th:1} and define  $g(z)=f(m+z)$ so that $g\in \mathbb{Z}[z]$ and $g=\sum_{i=0}^n s_i(m) z^i$, where $s_i(m)=f^{(i)}(m)/i!$ for each $i=0,\ldots,n$. By the hypothesis, we have $(m-1-h_f)^{\Delta_f}\geq d$ from which we have $m\geq h_f+1+\sqrt[\Delta_f]{d}$.  This in view of Lemma \ref{L:2} tells us that each zero of $g$ lies in the region $|z|> \sqrt[\Delta_f]{d}$ in the complex plane. This in view of rest of the hypotheses of the theorem and Lemma \ref{L:3} tells us that the polynomial $g$ (and hence $f$) is a product of at most $\min_{1\leq i\leq n}\{k,i+v_p(s_i(m))\}$ irreducible polynomials in $\mathbb{Z}[z]$.
\end{proof}
To prove Theorem \ref{th:4}, we first prove the following factorization result.
\begin{lemma}\label{L:5}
Let $g = s_0 + s_1z +\cdots+s_n z^n\in \Bbb{Z}[z]$ be a primitive polynomial with $s_0s_n\neq 0$ such that $s_n=\pm p^k d$ for some  positive integers $k$ and $d$ and a prime $p\nmid d$. Suppose that all zeros of $g$ lie in the region $|z|>\sqrt[\Delta_f]{d}$ in the complex plane.  If $|s_0/q|\leq |s_n|$, where $q$ is the smallest prime divisor of $s_0$, then the polynomial $g$ is a product of at most
\begin{eqnarray*}
\min_{1\leq i\leq n}\{k,i+v_p(s_{n-i})\}
\end{eqnarray*}
irreducible factors in $\mathbb{Z}[z]$. In particular, if $k=1$, or $p\nmid s_{n-1}$, then $g$ is irreducible.
\end{lemma}
In view of Theorem \ref{th:5} for the prime number $11$ and $j=n=6$, $d_1=2$, the polynomial
	\begin{eqnarray*}
		f=20449-3146z+121z^2+13442z^3-1034z^4+2209z^6
	\end{eqnarray*}
	has $\Delta_f=n/d_1=3$. Further, we have $s_6=47^2$ so that $p=47$, $k=2$, and $d=1$. Since $s_0=13^2\cdot11^2$, we have $q=11$, and so $|s_0/q|=1859<2209=|s_6|$. Now $\sqrt[\Delta_f]{d}=1$, and
	\begin{eqnarray*}
		\sum_{i=1}^{6}|s_i|=19952<20449=|s_0|,
	\end{eqnarray*}
	we deduce that each zero of $f$ lies inside the open region $|z|>1=\sqrt[\Delta_f]{d}$ in the complex plane. By Lemma \ref{L:5}, the polynomial $f$ is a product of at most
\begin{eqnarray*}
\min\{2,1+\infty,2+1,3+1,4+0,5+0,6+0\}=2
\end{eqnarray*}
irreducible factors in $\mathbb{Z}[z]$, where we note that $f$ is square of the irreducible polynomial $143-11z+47z^3$.
\begin{proof}[\bf Proof of Lemma \ref{L:5}]
Let $g(z)=g_1(z)\cdots g_N(z)$ be a product of $N$ irreducible polynomials $g_1,\ldots,g_N$ in $\mathbb{Z}[z]$. Then $\Delta_f\leq \deg g_i$ for all $i=1,\ldots,N$. Since $1\leq \min_{1\leq i\leq n}\{i+v_p(s_{n-i})\}$, we may assume without loss of generality that $N>1$. Let $\alpha_i\neq 0$ be the leading coefficient of $g_i$ for each $i$. We then have from the hypothesis that $\pm p^kd=s_n=\alpha_1\cdots \alpha_N$. We may write $g_i(z)=\alpha_i\prod_{\theta}(z-\theta)$ where the product runs over all zeros $\theta$ of $g_i$ for each $i$. We then arrive at the following:
\begin{eqnarray*}
|g_i(0)| &=& |\alpha_i|\prod_{\theta}|\theta|>  |\alpha_i|\bigl(\sqrt[\Delta_f]{d}\bigr)^{\deg g_i}\geq |\alpha_i|\bigl(\sqrt[\Delta_f]{d}\bigr)^{\Delta_f}=|\alpha_i|d\geq d,
\end{eqnarray*}
which shows that $|g_i(0)|>d$, and so, $|g_i(0)/q|\geq 1$ for all $i=1,\ldots, N$. Consequently, using the hypothesis, we get
\begin{eqnarray*}
|s_0/q| \leq |s_n|&<&|s_n|\left({\frac{|g_1(0)|}{|\alpha_1|d^{\frac{\deg g_1}{\Delta_f}}}\cdots \frac{|g_N(0)|}{|\alpha_N|d^{\frac{\deg g_N}{\Delta_f}}}}\right)\times \frac{|\alpha_i|d^{\frac{\deg g_i}{\Delta_f}}}{|g_i(0)|}
= |s_0|\frac{|\alpha_i|}{|g_i(0)|d^{\frac{n-\deg g_i}{\Delta_f}}},
\end{eqnarray*}
for each fixed index $i$, which yields $|\alpha_i|> |g_i(0)/q|d^{\frac{n-\deg g_i}{\Delta_f}}$ for each such $i$. We also
observe that
\begin{eqnarray*}
n-\deg g_i=\sum_{j=1,j\neq i}^{N} \deg g_j\geq \sum_{j=1,j\neq i}^{N}\Delta_f=(N-1)\Delta_f,
\end{eqnarray*}
for each index $i$, which in view of the fact that $N>1$ tells us that $(n-\deg g_i)/\Delta_f\geq N-1\geq 1$. This observation along with that $|g_i(0)/q|\geq 1$, we have
\begin{eqnarray*}
|\alpha_i|>|g_i(0)/q|d^{\frac{n-\deg g_i}{\Delta_f}}\geq d^{\frac{n-\deg g_i}{\Delta_f}}\geq d,~i=1,\ldots,N.
\end{eqnarray*}
Consequently, we have $|\alpha_i|>d$ for each $i$. This in the view that $p^k d=|s_n|=|\alpha_1|\cdots |\alpha_N|$ and the hypothesis that $p\nmid d$ shows that $p$ divides $|\alpha_i|$ for each $i=1,\ldots, N$. This proves that $N\leq k$.

Now let $j$ be a positive integer with $j\leq n$, and let $v_p(s_{n-j})=J\geq 0$. We may assume that $k > j+J$. Assume on the contrary that $N>j+J$, that is $N-j> J\geq 0$.  We may express $g_i=\sum_{t=0}^{\deg g_i}b_{it}z^t$ so that $\alpha_i=b_{i\deg g_i}$ for each $i$. Since we have $g(z)=g_1(z)\cdots g_N(z)$, we have $n=\sum_{i=1}^N \deg g_i$. Consequently, we have
\begin{eqnarray*}
s_{n-j}&=&\sum_{i_1+\cdots+i_N=n-j}b_{1i_1}\cdots b_{Ni_N}=\sum_{\sum_{t=1}^N(\deg g_t-i_t)=j} b_{1i_1}\cdots b_{Ni_N},
\end{eqnarray*}
where $0\leq i_t\leq \deg g_t$ and $0\leq \deg g_t -i_t\leq j$ for each index $i_t$. Since $N>j+J$, it follows that at least $N-j$ of the numbers $\deg g_1-i_1,\ldots,\deg g_N-i_N$ must be each equal to zero for any choice of the $N$-tuple $i_1,\ldots,i_N$. In view of this, for any $N$-tuple $i_1,\ldots, i_N$, at least $N-j$ of the indices satisfy $i_t=\deg g_t$, $1\leq t\leq N$, so that $b_{ti_t}=b_{t\deg g_t}=\alpha_t$, the leading coefficient of $g_t$, where we know that $p\mid \alpha_t$. Thus $p^{N-j}$ divides each term of the form $b_{1i_1}\cdots b_{Ni_N}$, and so, $p^{N-j}$ divides $\displaystyle \sum_{i_1+\cdots+i_N=n-j}b_{1i_1}\cdots b_{Ni_N}=s_{n-j}$, and so, $J=v_p(s_{n-j})\geq N-j>J$, which is a contradiction.
\end{proof}
\begin{proof}[\bf Proof of Theorem \ref{th:4}]
As in the proof of Theorem \ref{th:2}, we let $g(z)=f(m+z)$. So $g\in \mathbb{Z}[z]$ and $g=\sum_{i=0}^n s_i(m) z^i$, where $s_i(m)=f^{(i)}(m)/i!$ for each $i=0,1,\ldots,n$. Note that by the hypothesis we have $m\geq 1+h_f+\sqrt[\Delta_f]{d}$,  which in view of Lemma \ref{L:2} tells us that each zero of $g$ lies inside the open region $|z|>\sqrt[\Delta_f]{d}$ in the complex plane. Since $|s_n(m)|=p^k d$ and $|s_0(m)/q|\leq |s_n(m)|$, it follows that
the polynomial $g$ satisfies the hypothesis of Lemma \ref{L:5}. Consequently, the polynomial $g$ (and hence $f$) is a product of at most $\min_{1\leq i\leq n}\{k,i+v_p(s_{n-i}(m))\}$ irreducible polynomials in $\mathbb{Z}[z]$. Finally we observe that if we take $k=1$ or $p\nmid s_{n-1}(m)$, then $g$ is irreducible, and so, so does $f$.
\end{proof}
\subsection*{Acknowledgments}
%The authors are indebted to the referee for valuable comments and constructive criticism, which helped in improvement of presentation of the paper.
The Senior Research Fellowship (SRF) to Ms. Rishu Garg wide grant no. CSIRAWARD/JRF-NET2022/11769 from Council of Scientific and Industrial Research (CSIR), INDIA is gratefully acknowledged.
\subsection*{Disclosure statement}
The authors report to have no competing interests to declare.

\end{document}